\documentclass[11pt,reqno]{amsart}

\usepackage[english]{babel}
\usepackage{amsmath}
\usepackage{amsfonts}
\usepackage{amssymb}
\usepackage{amsthm}
\usepackage{graphicx}
\usepackage{epstopdf}
\usepackage{geometry}
\usepackage{hyperref}
\usepackage{tikz}
\usepackage[all]{xy}
\usepackage{multirow}
\usepackage{anysize}
\usepackage{verbatim}

\newtheorem{theorem}{Theorem}[section]

\newtheorem{example}[theorem]{Example}
\newtheorem{examples}[theorem]{Examples}
\newtheorem{proposition}[theorem]{Proposition}
\newtheorem{remark}[theorem]{Remark}
\newtheorem{corollary}[theorem]{Corollary}
\newtheorem{lemma}[theorem]{Lemma}

\theoremstyle{definition}
\newtheorem{definition}[theorem]{Definition}

\newcommand{\beq}{\begin{equation}}
\newcommand{\eeq}{\end{equation}}
\newcommand{\dhom}[3]{DH_{#1, #2}(#3)}

\newcommand{\dhomr}[2]{\dhom{#1}{r}{#2}}
\newcommand{\dhomi}[1]{\dhom{n}{\infty}{#1}}
\newcommand{\dhomn}[1]{\dhomr{n}{#1}}

\newcommand{\wdhom}[3]{\widetilde{DH}_{#1, #2}(#3)}

\newcommand{\wdhomr}[2]{\wdhom{#1}{r}{#2}}
\newcommand{\wdhomn}[1]{\wdhomr{n}{#1}}

\DeclareMathOperator{\id}{Id}

\title{Discrete Homology Theory for Metric Spaces}

\author{H\'el\`ene Barcelo}
\thanks{Address: Mathematical Sciences Research Institute, 17 Gauss Way, Berkeley, CA 94720}
\thanks{Email: hbarcelo@msri.org}

\author{Valerio Capraro}
\thanks{Address: Department of Mathematics, University of Southampton, Southampton, SO17 1BJ, UK}
\thanks{Email: V.Capraro@soton.ac.uk}

\author{Jacob A. White}
\thanks{Address Department of Mathematics, Texas A\& M University, Mailstop 3368, College Station, TX 77843-3368}
\thanks{Email:jwhite@math.tamu.edu}

\thanks{This material is based upon work supported by the National Science Foundation under Grant No. 0932078 000, 
while the first and second authors were in residence at the Mathematical Science Research Institute (MSRI) in Berkeley, California, during the Fall 2013 Semester.}

\keywords{Discrete Homotopy Theory, Coarse Geometry, Homology Theory}
\subjclass[2010]{Primary: 51F99, 55N35, Secondary: 05E99}

\begin{document}

\maketitle

\begin{abstract}
We define and study a notion of
discrete homology theory for metric spaces.
Instead of working with simplicial homology, our chain complexes are given by Lipschitz maps from an $n$-dimensional cube to a fixed metric space.
We prove that the resulting homology theory satisfies a discrete analogue of the Eilenberg-Steenrod axioms, and prove a discrete analogue of the Mayer-Vietoris exact sequence. 
Moreover, this discrete homology theory is related to the discrete 
homotopy theory of a metric space through a discrete analogue of the Hurewicz theorem. We study the class of groups that can arise as discrete homology groups and, in this setting, 
we prove that the fundamental group of a smooth, connected, metrizable, compact manifold is isomorphic to the discrete fundamental group of a `fine enough' rectangulation of the manifold. 
Finally, we show that this discrete homology theory can be coarsened, leading to a new non-trivial coarse invariant of a metric space.
\end{abstract}

\section{Introduction}\label{se:introduction}

Discrete homotopy theory is a discrete analogue of homotopy theory, associating a bigraded sequence of groups to a simplicial complex, capturing its combinatorial 
structure, rather than its topological structure. Originally called A-theory, it was developed in 
\cite{Kr-La98}, \cite{Ba-Kr-La-We01}, \cite{Ba-La05}, and \cite{Ba-Ba-Lo-La06}, 
which built on the work of Atkin \cite{At74},\cite{At76}.
 Discrete homotopy theory can be equivalently defined for finite connected graphs, 
resulting in an algebraic invariant of finite connected graphs and graph homomorphisms, in the same way that 
classical homotopy theory gives invariants of topological spaces and continuous maps: it associates a sequence $A_i(G)$ of groups to a finite connected graph. 
Discrete homotopy theory has been applied to different areas, 
including subspace arrangements \cite{Ba-Se-Wh11}, amenability of graphs \cite{Ca12a}, and pattern recognition \cite{Ca12b}. 
Discrete homotopy theory can also be generalized to arbitrary metric spaces, depending on a parameter $r$. The rationale for generalizing to arbitrary metric spaces 
will be given after Theorem \ref{th:lastintro}.

Is there a discrete homology theory that is related to discrete homotopy theory in the same way that 
classical homology is related to classical homotopy theory? The goal of this paper is to answer this question. 
We construct a collection of functors from the category of metric spaces and Lipschitz maps to the category of abelian groups. 
We define a discrete analogue of the Eilenberg-Steenrod axioms. 
Our first theorem is:

\begin{theorem}[Theorem \ref{th:eilenbergsteenrod}]
The discrete homology theory satisfies the discrete analogue of the Eilenberg-Steenrod axioms.
\end{theorem}
In the process, we define a notion of discrete cover, and prove an analogue of the Mayer-Vietoris exact sequence.
Continuing our analogy with classical homology, we prove the Hurewicz theorem in dimension 1.
\begin{theorem}[Theorem \ref{th:abelianization1}]
The abelianization of the discrete fundamental group of a metric space (at scale $r$) is isomorphic to its first discrete homology group (at scale $r$).
\end{theorem}

Then we study the class of groups that can arise as discrete homology groups of a metric space and we obtain the following proposition.
\begin{proposition}[Corollary \ref{cor:groups appearing as homology groups}]
For any abelian group $G$ and $\bar n\in\mathbb N$, there is a finite connected graph $\Gamma$ such that the $\bar n$-dimensional discrete homology group of $\Gamma$ (at scale $1$) is $G$ and all others are trivial.
\end{proposition}

It is worth mentioning that, in order to prove the previous proposition, we obtain a new result in discrete homotopy theory:
\begin{theorem}[Theorem \ref{th:discretization}]
\label{th:lastintro}
Let $R$ be a `fine enough' rectangulation of a compact, metrizable, smooth and path-connected manifold $M$, then the classical fundamental group of $M$ is isomorphic to the discrete fundamental group of $R$ at scale $1$.
\end{theorem}

A second motivation for the present paper comes from coarse geometry.
Informally speaking, coarse geometry is the study of a metric space from a large scale point of view.
Coarse algebraic topology is useful for approaching problems from different areas. For instance, 
Roe's coarse cohomology \cite{Ro93} was invented to perform index theory on non-compact manifolds and Higson and 
Roe's coarse homology \cite{Hi-Ro95} was used to formulate coarse versions of Baum-Connes's and Novikov's conjectures \cite{Yu95}. 
Other coarse homology theories have been discovered over the last two decades, including Block-Weinberger's \emph{uniformly finite homology} \cite{Bl-We92} 
and Nowak-\v Spakula's \emph{controlled coarse homology} \cite{No-Sp10}. 

The second goal of this paper is to construct a new coarse invariant obtained from the discrete homology groups.
Hence, instead of restricting ourselves to finite graphs, we work in the more general setting of metric spaces $X$ and we 
construct discrete homology groups depending on a scaling parameter $r$. 
Since a typical coarse object is a sequence of discrete objects we 
consider $\widetilde{DH}_{n,r}(X)$, the countably infinite direct product of copies of $\dhomn{X}$:
\begin{theorem}
The direct limit of $\widetilde{DH}_{n,r}(X)$, as $r\to\infty$, is a coarse invariant of $X$.
\end{theorem}
This coarse homology theory is a
relevant new object with many interesting open questions which we will study in the sequel to this paper.\\

The paper is organized as follows: in the next section, we review some terminology from discrete homotopy theory, and extend discrete homotopy theory to metric spaces. 
We also introduce discrete homology theory at scale $r$. In Section~\ref{se:eilenberg-steenrod}, 
we give discrete analogues of the Eilenberg-Steenrod axioms, and prove the discrete analogue of the Mayer-Vietoris exact sequence. 
The Eilenberg-Steenrod axioms allow us to relate discrete homology with discrete homotopy equivalence, in the same way that classical homology is 
related to classical homotopy equivalence. The results in this section also justify refering to our theory as `discrete homology' theory.
Then we study the discrete analogue of the Hurewicz map in Section~\ref{se:hurewicz}, which is followed by studying different constructions in 
Section~\ref{se:constructions}. In Section~\ref{se:coarse}, we define the discrete coarse homology theory. Finally, we end the paper with some 
conclusions, and open problems.\\

\textbf{Acknowledgements.} We thank Eric Babson, Dennis Dreesen, Antoine Gournay, Thibault Pillon and Nick Wright for helpful discussions.

\section{Discrete homotopy theory and Cubical discrete homology theory}\label{se:homology}

First, we recall definitions from discrete homotopy theory of graphs, and extend them to metric spaces.
Given two metric spaces $(X,d_X)$, $(Y,d_Y)$ and $r>0$, a function $f:X\to Y$ is \emph{$r$-Lipschitz} if $d_Y(f(x_1),f(x_2))\leq rd_X(x_1,x_2)$, for all $x_1,x_2\in X$.
When $X$ is the vertex set of a connected graph $G$, there is a well-defined metric $d_G$, the length of the shortest path connecting two vertices. For connected graphs, the definition of graph 
homomorphism $f: G \to H$ appearing in \cite{Ba-Kr-La-We01} is equivalent to saying that the map $f$ is $1$-Lipschitz between the corresponding metric spaces $(V(G), d_G)$ and $(V(H), d_H)$.

In the following definition, $\{0,\ldots m\}$ is equipped with the metric $d(a,b)=|a-b|$ and the cartesian product $X\times \{0, \ldots m \}$ is equipped with the $\ell^1$-metric.
\begin{definition}
 Let $X, Y$ be metric spaces, and  let $f,g:X \to Y$ be $r$-Lipschitz maps. Then $f$ and $g$ are $r$-discrete homotopic if there exists a non-negative integer $m$ and an
 $r$-Lipschitz map $F: X \times \{0, \ldots m \} \to Y$ such that $F(-, 0) = f$ and $F(-, m) = g$.
\end{definition}
This definition generalizes the notion of discrete homotopic maps appearing in \cite{Ba-Kr-La-We01}. Since our definition of two $r$-Lipschitz maps being `homotopic' involves $\{0, \ldots, m \}$, 
a finite discrete metric space, we continue to call the resulting theory discrete homotopy theory.

Now we define the discrete fundamental group of a metric space (at scale $r$).
Let $(X,d)$ be a metric space. An \emph{r-path connecting x to y} is a finite sequence of points $x_0x_1\ldots x_nx_{n+1}$ such that $x_0=x$, $x_{n+1}=y$ and $d(x_i,x_{i+1})\leq r$, for all $i$. 
Fix a base point $p\in X$, an \emph{r-loop based at p} is an $r$-path such that $x_0=x_{n+1}=p$. Denote by $\mathcal C_r(X,p)$ the set of $r$-loops based at $p$.
The set $\mathcal C_r(X,p)$ is a group under the operation of concatenation of $r$-loops, that is,
$$
(x_0x_1\ldots x_{m}x_{m+1})(y_0y_1\ldots y_{n}y_{n+1}):=x_0x_1\ldots x_{m}x_{m+1}y_0y_1\ldots y_{n}y_{n+1}.
$$

\begin{definition}\label{defin:r-homotopy}
Every $r$-loop $x_0 x_1 \ldots x_n x_{n+1}$ is $r$-homotopy equivalent to $x_0 x_1 \ldots x_{n+1} p$. Moreover, given
 two $r$-loops $x_0 x_1 \ldots x_n x_{n+1}$ and $y_0y_1\ldots y_n y_{n+1}$ of the same length, then the $r$-loops are $r$-homotopy equivalent if there is a formal matrix
$$
\left(
  \begin{array}{ccccc}
    x_0 & x_1 & \ldots & x_n & x_{n+1} \\
    z_0^1 & z_1^1 & \ldots & z_{n}^1 & z_{n+1}^1 \\
    \vdots & \vdots & \ldots & \vdots & \vdots \\
    z_0^t & z_1^{t} & \ldots & z_{n}^{t} & z_{n+1}^t \\
    y_0 & y_1 & \ldots & y_{n}& y_{n+1} \\
  \end{array}
\right)
$$
where
\begin{itemize}
\item every row is an $r$-loop based at $p$,
\item every column is an $r$-path based at $p$.
\end{itemize} 
\end{definition}
The $r$-homotopy equivalence, denoted by $\sim_r$, is an equivalence relation on $\mathcal C_r(X,p)$ which agrees with the operation of concatenation.

\begin{definition}\label{defin:fundamental group at scale r}
The discrete fundamental group at scale $r$ is
$$
A_{1,r}(X,p):=\mathcal C_{r}(X,p)/\sim_r.
$$
\end{definition}

\begin{definition}
 A metric space $(X,d)$ is \emph{connected at scale $r>0$} if for all $x,y\in X$, there are $x_0,x_1,\ldots,x_{n},x_{n+1}\in X$ such that $x_0=x$, $x_{n+1}=y$, and $d(x_i,x_{i+1})\leq r$, for all $i$.
\end{definition}

If $(X,d)$ is connected at scale $r$, then this group does not depend on $p$. So it will be denoted by $A_{1,r}(X)$. 
Observe that if $X$ is a finite connected graph and $r=1$, then $A_{1,1}(X)$ is the discrete fundamental group of the graph as defined in \cite{Ba-Kr-La-We01}.
There are higher homotopy groups $A_{n,r}(X)$. For $r = 1$, and when $X$ is a graph, they are defined in \cite{Ba-Kr-La-We01}.
Alternatively, discrete homotopy theory is defined as a cubical homotopy theory in \cite{Ba-Ba-Lo-La06}.

Now we define our discrete homology theory. Note that discrete homotopy theory involves formal matrices, which can be thought of as a finite subspace of the metric space $\mathbb{Z}^2$. 
Higher discrete homotopy groups also involve finite subspaces of $\mathbb{Z}^n$, which can be thought of as a cubical subdivision of a cube.
With this motivation, we define our discrete homology groups in terms of Lipschitz maps from discrete cubes.

Let $Q_n$ denote the $n$-dimensional cube, represented by the set of points $\{(a_1,\ldots,a_n)\subseteq\mathbb R^n : a_i\in\{0,1\}\}$, equipped with the Hamming distance \cite{Ha50}.
\begin{definition}
A \emph{singular $(n,r)$-cube} is an $r$-Lipschitz map $\sigma : Q_n\to X$.
\end{definition}

Let $L_{n,r}(X)$ be the free abelian group generated by all singular $(n,r)$-cubes.
Let $\sigma : Q_n\to X$ be an $(n,r)$-cube, with $n\geq1$. Let $i\in\{1,\ldots,n\}$.
\begin{itemize}
\item The $i$-th front face of $\sigma$ is the singular $(n-1,r)$-cube
\[
(A_i^n\sigma)(a_1,\ldots,a_{n-1})=\sigma(a_1,\ldots,a_{i-1},0,a_i,\ldots,a_n),
\]
\item The $i$-th back face of $\sigma$ is the singular $(n-1,r)$-cube
\[
(B_i^n\sigma)(a_1,\ldots,a_{n-1})=\sigma(a_1,\ldots,a_{i-1},1,a_i,\ldots,a_n).
\]
\end{itemize}

A singular $(n,r)$-cube $\sigma$ is \emph{degenerate} if $A_i^n \sigma = B_i^n \sigma$ for some $i$. Let $D_{n,r}(X)$ be the free abelian group generated by all degenerate $(n,r)$-cubes 
and $C_{n,r}(X) = L_{n,r}(X) / D_{n,r}(X)$, whose elements are called $(n,r)$-chains in $X$.

\begin{definition}
 The \emph{boundary} of a singular $(n,r)$-cube $\sigma : Q_n\to X$ is
\[
\partial \sigma:=\sum_{i=1}^n(-1)^{i}\left(A_i^n\sigma-B_i^n\sigma\right).
\]
\end{definition}

The boundary operator extends to a group homomorphism $\partial_n : L_{n,r}(X)\to L_{n-1,r}(X)$. Since $\partial_n(D_{n,r}(X)) \subseteq D_{n-1,r}(X)$, we obtain a map $\partial_n: C_{n,r}(X)\to C_{n-1,r}(X)$ 
and a chain complex $(\mathcal{C}, \partial)$. 
\begin{definition}
Let $\dhomr{\ast}{X}$ denote the homology groups of the resulting chain complex; these are the \emph{discrete homology groups at scale $r$}.
\end{definition}

More precisely, define $B_{n,r}(X):=\text{Im}(\partial_{n+1})$, the \emph{group of (n,r)-boundaries}, and $Z_{n,r}(X):=\text{ker}(\partial_{n})$, the \emph{group of (n,r)-cycles}. 
Since $\delta_n\circ\delta_{n+1}=0$, one has $B_{n,r}(X)\subseteq Z_{n,r}(X)$ and
$$
\dhomn{X}:=Z_{n,r}(X)/B_{n,r}(X).
$$

Similarly, let $A \subset X$ be a subspace of $X$. Then one can show that $\partial_n(C_{n,r}(A)) \subseteq C_{n-1,r}(A)$, so that there is a 
map $\partial_n: C_{n,r}(X, A) \to C_{n-1, r}(X, A)$, where $C_{n,r}(X,A) = C_{n,r}(X) / C_{n,r}(A)$. The resulting homology of this chain complex is 
the \emph{relative} homology $\dhomr{\ast}{X, A}$.

\section{Discrete Analogue of the Eilenberg-Steenrod Axioms} \label{se:eilenberg-steenrod}

The purpose of this section is to show that the discrete homology theory $\dhomr{*}{X}$ satisfies a discrete analogue of the classical Eilenberg-Steenrod axioms \cite{Ei-St45} 
and a discrete analogue of the classical Mayer-Vietoris exact sequence \cite{Ma29},\cite{Vi30}. In order to prove these results, we need a notion of `discrete open cover'. 
As motivation for our definition, observe that, for the classical results, the notion of open set is not important in itself, but rather the following geometric property: 
that a simplex can be replaced with a sum of simplices, which each lie in the interior of some open set in the cover.
This is our motivation for the following discrete analogue of cover.

\begin{definition}\label{defin:n-dimensional cover}
 An \emph{n-dimensional discrete cover} at scale $r>0$ of a metric space $X$ is a family $\mathcal U=\{U_i\}$ of subsets of $X$ such that:
\begin{enumerate}
\item $X=\bigcup U_i$,
\item For all non-degenerate $(n,r)$-cubes $\sigma$, $\sigma\in U_i$ for some $i$.
\end{enumerate}

\end{definition}

Two dimensional discrete covers (for graphs and with $r=1$) have already appeared in \cite{Ba-Kr-La-We01} to prove the Seifert-van Kampen Theorem for discrete fundamental groups. 

\begin{remark}\label{rem:2-dimension does not imply n-dimension}
{\rm Two dimensional discrete covers are not necessarily $3$-dimensional discrete covers. Take the 3-dimensional Hamming cube $Q_3$, 
and consider the cover made of its six 2-dimensional faces. This is a 2-dimensional cover, but not a 3-dimensional cover.}

\end{remark}

\begin{definition}\label{defin:discrete open cover}
 A \emph{discrete cover} at scale $r$ is an $n$-dimensional discrete cover, for all $n$.
\end{definition}

Given subsets $A,B$ of $X$, consider the canonical inclusions $i^1:A \cap B \to A$, $i^2:A \cap B \to B$, $j^1:A \to A \cup B$, $j^2: B \to A \cup B$. \\

\begin{theorem}[Mayer-Vietoris sequence]
Let $X$ be a metric space and let $\{A, B\}$ be a discrete cover of $X$ at scale $r > 0$. Then we have the following long exact sequence:

\xymatrix{
 \cdots \ar[r]^-{\partial_*} & \dhomn{A \cap B}\ar[r]^-{diag} & \dhomn{A} \oplus \dhomn{B} \ar[r]^-{diff} & \dhomn{A \cup B} \ar[r]^-{\partial_*} & \dhomr{n-1}{A \cap B}\ar[r]^-{diag} & \cdots 
}

where diag = $(i^1_*, i^2_*)$ and  diff = $j^1_* - j^2_*$.

\end{theorem}
\begin{proof}
 Consider the following sequence of chain complexes:
\[
\xymatrix@1{
 0 \ar[r] & C_{*,r}(A \cap B)\ar[r]^-{\mbox{diag}} & C_{*,r}(A) \oplus C_{*,r}(B) \ar[r]^-{\mbox{diff}} & C_{*,r}(A \cup B) \ar[r] & 0 
}
\]
By the zig-zag lemma, if the above sequence is short exact, then we obtain a long exact sequence in homology. Injectivity of diag, and exactness at $C_{n,r}(A) \oplus C_{n,r}(B)$ are both clear.
 To prove surjectivity of diff, consider a singular $(n,r)$-cube $\sigma$ in $A \cup B$. Since $\{A, B\}$ is a discrete cover, then $\sigma$ lies in $A$ or $B$, and hence in the image of diff.
\end{proof}
Now we give a counterexample to the Mayer-Vietoris Theorem when $\{A, B\}$ is not a discrete cover. 

\begin{example}\label{ex:no MV}
 {\rm Consider the graph $C_5$, with vertex set $\{x_1,\ldots,x_5\}$, where $x_i,x_{i+1}$ and $x_0,x_5$ are adjacent.
 Let $A=\{x_1\}$ and $B=\{x_2,\ldots,x_5\}$. Then $H_{n,1}(A)=H_{n,1}(B)=H_{n,1}(A\cap B)=\{0\}$, for all $n$, hence $H_{1,1}(C_5)=\{0\}$, if 
the Mayer-Vietoris Theorem holds for $\{A, B \}$. However, Theorem \ref{th:abelianization1} can be used to show that $H_{1,1}(C_5) = \mathbb{Z}$. }
\end{example}

\begin{definition}
Let $Met$ be the category whose objects are triples $(X,Y, r)$ where  $X$ is a metric space, $Y$ is a subspace of $X$, and $r$ is a real number. The morphisms are pairs $(f,k): (A,B,r) \to (X,Y,kr)$ where 
$f: A \to Y$ is $k$-Lipschitz, and $f(B) \subseteq Y$.
\end{definition}

We can now formulate the discrete analogue of the Eilenberg-Steenrod axioms.

\begin{definition}
 A \emph{discrete homology theory at scale $r$} consists of:
\begin{itemize}
 \item A sequence $\dhomn{-}$ of functors from $Met$ to the category of abelian groups.
 \item A natural transformation $\partial: \dhomn{X,A} \to \dhomr{n-1}{A}$
\end{itemize}
 These functors are subject to the following axioms:
\begin{enumerate}
\item (homotopy) Discrete homotopic maps induce the same map in discrete homology. That is, if $f:(X,A)\to(Y,B)$ is homotopic to $g:(X,A)\to(Y,B)$, then their induced maps are the same.
\item (excision) Let  $\{A, B\}$ be a discrete cover of $X$. Then the inclusion $(B, A\cap B)\to (X, A)$ induces isomorphisms $\dhomn{B,A\cap B} \to \dhomn{X,A}$, for all $n$.
\item (dimension) Let $P$ be the one-point space; then $\dhomn{P,\emptyset}=\{0\}$, for all $n\geq1$ and for all $r>0$.
\item (long exact sequence) Each pair $(X,A)$ induces a long exact sequence via the inclusion maps $i:A\to X$ and $j:(X,\emptyset)\to(X,A)$:

\[
\xymatrix@1{
 \ldots \ar[r] & \dhomn{A}\ar[r]^-{i_*} & \dhomn{X} \ar[r]^-{j_*} & \dhomn{X,A} \ar[r]^-{\partial_*} & \dhomr{n-1}{A}\ar[r] & \dots 
}
\]
\end{enumerate}

\end{definition}

We have defined the functors $\dhomr{\ast}{X,A}$. The natural transformations $\partial: \dhomn{X,A} \to \dhomr{n-1}{A}$ is defined as follows: Let $\sigma$ be a representative of a class in $\dhomn{X}{A}$, and let 
$\tau$ be an $n$-chain in $X$ which maps onto $\sigma$ (under the quotient map $C_{n,r}(X) \mapsto C_{n,r}(X) / C_{n,r}(A)$). Then $\partial([\sigma]) = [ \partial(\sigma)]$. One can check that $\partial(\sigma) \in C_{n-1,r}(A)$, 
and that the map is well-defined.
\begin{theorem}
\label{th:eilenbergsteenrod}
The functors
$\dhomr{\ast}{X,A}$ form a discrete homology theory.
\end{theorem}

\begin{proof}

We prove the result item-by-item.

\begin{enumerate}

\item (homotopy) Suppose $f$ and $g$ are $k$-Lipschitz, and discrete homotopic. Let $F: X \times [m] \to Y$ be a discrete homotopy from $f$ to $g$. 
For $\sigma \in C_{n,r}(X)$, let $\Phi(\sigma) \in C_{n+1}(Y) = \sum_{i=0}^{m-1} \sigma_i$, where $\sigma_i$ is the $(n+1,kr)$-cube 
such that $A_1 \sigma_i = F(\sigma, i)$ and $B_1 \sigma_i = F(\sigma, i+1)$. The map $\Phi$ is a chain homotopy from $C_{\star, r}(X)$ to $C_{\star, kr}(Y)$; the result follows from homological algebra.

\item (Excision) Let $\mathcal U=\{A,B\}$ and let $C_{n,r}^{\mathcal U}(X)$ to be the subgroup of $C_{n,r}(X)$ generated by singular (n,r)-cubes with image in $A$ or in $B$. Then $C_{n,r}^{\mathcal U}(X)=C_{n,r}(X)$, 
for all $n$, since $\mathcal U$ is a discrete cover. Hence we have an isomorphism in homology
\begin{align}\label{eq:iso}
H^{\mathcal U}(X)\simeq \dhomn{X}.
\end{align}
Now observe that the map
$$
C_{n,r}(B)/C_{n,r}(A\cap B)\to C_{n,r}^{\mathcal U}(X)/C_{n,r}(A), 
$$
induced by inclusion, is obviously an isomorphism since both quotient groups are free with basis the singular (n,r)-cubes in B which do not lie in A. Hence, using (\ref{eq:iso}), we have an isomorphism in homology
$$
\dhomn{B,A\cap B}\simeq \dhomn{X,A}.
$$ 
\item (dimension) Observe that, for $n\geq1$ and $r>0$, every singular $(n,r)$ cube is degenerate and so we have $C_{n,r}=\{0\}$. Consequently, $\dhomn{P}$ is trivial for all $n\geq1$ and for all $r>0$.
\item (long exact sequence) As in the classical case, the maps $i$ and $j$ give rise to a short exact sequence of chain complexes; the result follows from the zig-zag lemma. Note that the zig-zag lemma gives the existence and uniqueness of 
the natural transformation $\partial: \dhomn{X,A} \to \dhomr{n-1}{A}$.
\end{enumerate}
\end{proof}

\section{Hurewicz theorem in dimension one} \label{se:hurewicz}

The aim of this section is to prove an analogue of the Hurewicz theorem for discrete homology theory;
$\dhomr{1}{X}$ is isomorphic to the abelianization of the discrete fundamental group of $X$ at scale $r$.
Definitions \ref{defin:r-homotopy} and \ref{defin:fundamental group at scale r} are used in this section.

\begin{theorem}\label{th:abelianization1}
Let $(X,d)$ be a metric space which is connected at scale $r$. Then $\dhomr{1}{X}$ is isomorphic to the abelianization of $A_{1,r}(X)$.
\end{theorem}

\begin{proof}
We construct a map
$$
\phi:A_{1,r}(X,p)\to \dhomr{1}{X}
$$
such that
\begin{enumerate}
\item $\phi$ is a group homomorphism,
\item $\phi$ is surjective,
\item $\text{Ker}(\phi)=[A_{1,r}(X,p),A_{1,r}(X,p)]$, the commutator subgroup of $A_{1,r}(X,p)$.\\
\end{enumerate}

\textbf{Step 1. Definition of $\phi$.}
Let $[\gamma]\in A_{1,r}(X,p)$ and choose a representative $\gamma=x_0x_1\ldots x_nx_{n+1}$, with $x_0=x_{n+1}=p$, such that $x_i\neq x_{i+1}$, for all $i$. 
This condition and the fact that $d(x_i,x_{i+1})\leq r$ imply that the maps $\sigma_i:Q_1\to X$, defined as $\sigma_i(0)=x_i$, $\sigma_i(1)=x_{i+1}$, are non-degenerate $(1,r)$-cubes, 
for all $i$. Therefore $\sum_i\sigma_i\in C_{1,r}(X)$. Now observe that
$$
\partial\left(\sum_i\sigma_i\right)=(x_1-x_0)+(x_2-x_1)+\ldots+(x_{n+1}-x_n)=0,
$$
since $x_0=x_{n+1}$. Consequently $\sum_i\sigma_i\in Z_{1,r}(X)$. Denote by $\overline\gamma$ the image of $\sum_i\sigma_i$ down in $\dhomr{1}{X}$ and define $\phi([\gamma])=\overline\gamma$.\\

\textbf{Step 2. $\phi$ is well defined.}
Let $\gamma_1=px_1x_2\ldots x_mp$ and $\gamma_2=py_1y_2\ldots y_np$ be two $r$-homotopy equivalent $r$-loops. Therefore, we have a homotopy matrix
$$
\left(
  \begin{array}{ccccc}
    p & x_1 & \ldots & x_s & p \\
    p & z_1^1 & \ldots & z_{s}^1 & p \\
    \vdots & \vdots & \ldots & \vdots & \vdots \\
    p & z_1^{t} & \ldots & z_{s}^{t} & p \\
    p & y_1 & \ldots & y_{s}& p \\
  \end{array}
\right)
$$
Observe that all \emph{squares} $\{p,p,z_1^1,x_1\}, \{x_1,z_1^1,z_2^1,x_2\}, \ldots$ forming this homotopy matrix are in fact singular $(2,r)$-cubes. Consider 
the $(2,r)$-chain $\tau$ given by summing over all the squares appearing in the matrix.
To compute the boundary, observe that the sides that belong inside the matrix are ran once in one 
direction and once in the other, so that the boundary of $\tau$ is equal to the boundary of the matrix, which is $\gamma_1\gamma_2^{-1}$. 
In other words, $\overline\gamma_1=\overline\gamma_2$.\\

\textbf{Step 3. $\phi$ is a group homomorphism.}
This follows by direct computation. Indeed, if $\gamma_1=px_1x_2\ldots x_mp$ and $\gamma_2=py_1y_2\ldots y_np$. One has
$$
\overline\gamma_1=\sum_i\sigma_i,\qquad\qquad\overline\gamma_2=\sum_i\mu_i,
$$
where $\sigma_i(0)=x_i, \sigma_i(1)=x_{i+1}, \mu_i(0)=y_i, \mu_i(1)=y_{i+1}$. Consequently,
\begin{align}\label{eq:first}
\overline\gamma_1+\overline\gamma_2=\sum_i\sigma_i+\sum_i\mu_i.
\end{align}
On the other hand, $\gamma_1\gamma_2=px_1\ldots x_mpy_1\ldots y_np$. Therefore,
$$
\overline{\gamma_1\gamma_2}=\sum_i\nu_i,
$$
where $\nu_i=\sigma_i$, for $i\in\{0,\ldots,m\}$, and $\nu_{m+i}=\mu_i$, for $i=\{0,\ldots,n\}$. This is then clearly equal to (\ref{eq:first}).\\

\textbf{Step 4. $\phi$ is surjective.}
Let $\lambda$ be a singular $(1,r)$-cycle, say $\lambda=\sum_{i=1}^kn_i\sigma_i$, where $\sigma_i:\{0,1\}\to X$ are such that $d(\sigma_i(0),\sigma_i(1))\leq r$, for all $i$. Since $\partial\lambda=0$, one has
\begin{align}\label{eq:boundary}
\sum_{i=1}^kn_i\left(\sigma_i(1)-\sigma_i(0)\right)=0.
\end{align}
Let $S:=\{\sigma_i(0),\sigma_i(1) : i=1,\ldots,k\}$. For a given $q\in S$, 
\begin{itemize}
\item let $m_q$ be the sum of the coefficients of $q$ in (\ref{eq:boundary}). Observe that $m_q=0$, by Equation (\ref{eq:boundary}).
\item let $\beta_q$ be an $r$-path connecting $p$ and $q$.
\end{itemize}
For all $i=1,\ldots,k$, define the $r$-loop $\eta_i:=\beta_{\gamma_i(0)}\sigma_i\beta_{\gamma_i(1)}$ and define 
$$
\gamma:=\eta_1^{n_1}\eta_2^{n_2}\ldots\eta_k^{n_k}.
$$
We now show that $\phi(\gamma)=\lambda$. Indeed,
\begin{align*}
\phi(\gamma)\\
&=\sum_{i=1}^kn_i\sigma_i-\sum_{i=1}^kn_i(\beta_{\gamma_i(1)}-\beta_{\gamma_i(0)})\\
&=\sum_{i=1^k}n_i\sigma_i-\sum_{q\in S}m_q\beta_q\\
&=\sum_{i=1}^kn_i\sigma_i\\
&=\lambda.
\end{align*}
\textbf{Step 5. The kernel of $\phi$ is equal to the commutator subgroup of $A_{1,r}$.}
Since $\text{Im}(\phi)$ is an abelian subgroup of $\dhomr{1}{X}$, then $[A_{1,r},A_{1,r}]\subseteq \text{ker}(\phi)$. 
To prove the other inclusion, let $\gamma$ be an $r$-loop based at $p$ whose homotopic class $[\gamma]$ belongs to $\text{ker}(\phi)$. 
Since $\overline\gamma=0$, then $\gamma$ is the boundary of a $(2,r)$-chain $\sigma=\sum_{i=1}^kn_i\sigma_i$, where 
$$
\sigma_i:\{(0,0),(1,0),(0,1),(1,1)\}\to X
$$
are such that the distance of the images of two adjacent vertices is at most $r$. Now, write $\partial\sigma_i$ as the sum of its faces; that is,
$$
\partial\sigma_i=a_i^++b_i^++a_i^-+b_i^-,
$$
where
\begin{align}\label{eq:faces1}
a_i^+(0)=b_i^-(1)=\sigma_i(0,0),\qquad\qquad a_i^+(1)=b_i^+(0)=\sigma_i(1,0),
\end{align}
\begin{align}\label{eq:faces2}
a_i^-(0)=b_i^+(1)=\sigma_i(1,1),\qquad\qquad a_i^-(1)=b_i^-(0)=\sigma_i(0,1).
\end{align}
We have
\begin{align}\label{eq:identity}
\gamma=\partial\left(\sum_in_i\sigma_i\right)=\sum_in_i(a_i^++b_i^++a_i^-+b_i^-).
\end{align}
Fix the following notation. Let
$$
L:=\{(a_i^+,b_i^+,a_i^-,b_i^-) : i=1,\ldots,k\},
$$
and let $S$ be the set of endpoints of the $(1,r)$-chain $a_i^+,b_i^+,a_i^-$, and $b_i^-$. For all $q\in S$, let $\eta_q$ be an $r$-path joining $p$ and $q$. 
For all $\theta\in L$, let $m_\theta$ be the sum of the coefficients of $\theta$ in Equation (\ref{eq:identity}).
Now, for each singular $(2,r)$-cube, consider the following four $(2,r)$-paths:
$$
\beta_{a_i^+(0)}a_i^+\beta_{a_i^+(1)}^{-1},\qquad\qquad \beta_{b_i^+(0)}b_i^+\beta_{b_i^+(1)}^{-1},
$$
$$
\beta_{a_i^-(0)}a_i^-\beta_{a_i^-(1)}^{-1},\qquad\qquad \beta_{b_i^-(0)}b_i^-\beta_{b_i^-(1)}^{-1}.
$$
Denote by $\eta_i$ the concatenation of all these paths. Using Equations (\ref{eq:faces1}) and \ref{eq:faces2}, we can show that $\eta_i$ is $r$-homotopy equivalent to the constant loop. Indeed
\begin{align*}
\eta_i & = & \beta_{a_i^+(0)}a_i^+\beta_{a_i^+(1)}^{-1}\beta_{b_i^+(0)}b_i^+\beta_{b_i^+(1)}^{-1}\beta_{a_i^-(0)}a_i^-\beta_{a_i^-(1)}^{-1}\beta_{b_i^-(0)}b_i^-\beta_{b_i^-(1)}^{-1}\\
& = & \beta_{a_i^+(0)}a_i^+\beta_{a_i^+(1)}^{-1}\beta_{a_i^+(1)}b_i^+\beta_{b_i^+(1)}^{-1}\beta_{b_i^+(1)}a_i^-\beta_{a_i^-(1)}^{-1}\beta_{a_i^-(1)}b_i^-\beta_{b_i^-(1)}^{-1}\\
& = & \beta_{a_i^+(0)}a_i^+b_i^+a_i^-b_i^-\beta_{a_i^+(0)}^{-1}.
\end{align*}
To see that this latter loop is homotopy equivalent to the constant loop, write first $\beta_{a_i^+(0)}=px_1x_2\ldots x_m\sigma_i(0,0)$, so as we have
\begin{align*}
\eta_i & = 
& px_1x_2\ldots x_m\sigma_i(0,0)\sigma_i(1,0)\sigma_i(1,1)\sigma_i(0,1)\sigma_i(0,0)x_mx_{m-1}\ldots x_2x_1p \\
& \equiv & px_1x_2\ldots x_m\sigma_i(0,0)\sigma_i(1,0)\sigma_i(1,0)\sigma_i(0,0)\sigma_i(0,0)x_mx_{m-1}\ldots x_2x_1p\\
& \equiv & px_1x_2\ldots x_m\sigma_i(0,0)\sigma_i(0,0)\sigma_i(0,0)\sigma_i(0,0)\sigma_i(0,0)x_mx_{m-1}\ldots x_2x_1p\\
& \equiv & p.
\end{align*}
Now $\gamma = (\gamma_0, \ldots, \gamma_k, \gamma_0)$. Let $\tau_i$ be the singular $(1,r)$-cube such that $\tau_i(0) = \gamma_i$ and $\tau_i(1) = \gamma_{i+1}$. 
Let $\omega$ be the concatenation $\prod_{i=0}^k\beta_{\tau_i(0)} \tau_i \beta_{\tau_i(1)}^{-1}$, which is homotopy equivalent to $\gamma$.
Now set $\delta=\eta_1^{n_1}\eta_2^{n_2}\ldots\eta_k^{n_k}$, and we see that $\delta \omega^{-1}$ is homotopy equivalent to $\gamma^{-1}$. 
However, for each $\theta \in L$, the loop $\beta_{\theta(0)} \theta \beta_{\theta(1)}^{-1}$ appears $m_{\theta}$ times in $\delta$, and $-m_{\theta}$ times 
in $\omega^{-1}$. Thus each such loop appears in $\delta \omega^{-1}$ with exponent summing to $0$; hence $\delta \omega^{-1}$ map to the identity in the abelianization 
of $A_{1,r}(X)$. Thus $[\gamma]$ is in the commutator of $A_{1,r}(X)$. 
\end{proof}

\section{Groups appearing as homology groups}\label{se:constructions}

In this section, we describe groups which appear as discrete homology groups (and discrete homotopy groups). 
Our main result is Corollary \ref{cor:groups appearing as homology groups}, that is based on facts of intrinsic interest: 
the classical fundamental group of a compact, smooth, metrizable, path-connected manifold is the discrete 
fundamental group of a metric space (and thus the first homology group is also a discrete homology group). We also give a construction which is an analogue of suspension.

Let $M$ be a compact, metrizable, smooth, and path connected manifold. Let $\ell>0$ be the infimum of the lengths of all non-contractible loops on $M$. 
We say that a rectangulation $R$ of the manifold is `fine enough' if each side of the rectangulation has length at most $\ell/5$.

\begin{theorem}\label{th:discretization}
Let $R$ be a `fine enough' rectangulation of a compact, metrizable, smooth, and path-connected manifold $M$. Let $X$ be the natural graph associated to 
the rectangulation $R$. Then $\pi_1(M) \simeq A_{1,1}(X)$.
\end{theorem}

\begin{proof}

Let $p\in X$ be a base point. We can associate to every loop in $M$ based at $p$ a discrete 
loop in $X$ by the following procedure. Let $\gamma:[0,1]\rightarrow M$ be a continous loop in $M$ based at $p$. For every $t\in[0,1]$ associate to $\gamma(t)$ one of its closest points in $X$. 
By continuity of $\gamma$, one can do this in such a way to obtain a 1-loop $\tilde\gamma=x_0x_1\ldots x_n$ in $R$. This procedure preserves the homotopy equivalence. Indeed, let $H:[0,1]\times[0,1]\rightarrow M$ 
be an homotopy between loops $\gamma_0$ and $\gamma_1$ in $M$. Set $\widetilde H_{\frac{i}{N}}:=\widetilde H\left(\frac{i}{N},t\right)$. Since $M$ is a 
compact metric space, by the Heine-Cantor theorem, one can find $N>0$ such that $\widetilde{H}_\frac{i}{N}$ and $\widetilde {H}_{\frac{i+1}{N}}$ are 
consecutive steps of a discrete homotopy in $R$, for all $i$.The sequence $\widetilde\gamma_0\widetilde {H}_{\frac 1N}\widetilde {H}_{\frac 2N}\ldots \widetilde {H}_{1-\frac1N}\widetilde\gamma_1$ 
then gives rise to a homotopy between $\widetilde{\gamma}_0$ and $\widetilde{\gamma}_1$.

Note that the converse does not hold a priori for an arbitrary triangulation and this is why we need a \emph{fine enough} rectangulation. For such rectangulations, we have that if $\tilde\gamma_0=px_1\ldots x_sp$ and $\tilde\gamma_1=py_1\ldots y_sp$ 
are discrete homotopy equivalent, then $\gamma_0$ and $\gamma_1$ are homotopy equivalent. 
Indeed, let  $\widetilde{\gamma}_0$ and $\widetilde{\gamma}_1$ be two loops in $X$ that are homotopy equivalent and consider a homotopy matrix between them.

$$
\left(
  \begin{array}{ccccc}
    p & x_1 & \ldots & x_s & p \\
    p & z_1^1 & \ldots & z_{s}^1 & p \\
    \vdots & \vdots & \ldots & \vdots & \vdots \\
    p & z_1^{t} & \ldots & z_{s}^{t} & p \\
    p & y_1 & \ldots & y_{s}& p \\
  \end{array}
\right)
$$

Now, to every row and every column of the homotopy matrix, associate the path in $M$ that is obtained by 
connecting adjacent points in the homotopy matrix via the edge in the rectangulation connecting them. 
Because no rectangule of the rectangulation can contain a non-contractible loop, we know that the loop associated this way to the first row is 
homotopy equivalent to $\gamma_0$ and the loop associated to the last row is homotopy equivalent to $\gamma_1$. Moreover, we know that all \emph{squares} $ppz_1^1x_1p$, $x_1z_1^1z_2^1x_2^1x_1^1$, etc. 
do not contain any contractible loop in $M$, because of our choice of the rectangulation. In other words, all these squares, seen as loops in $M$ as above, are contractible and therefore, 
they can be filled by a standard homotopy. Gluing all these homotopies, one obtain a standard homotopy between the loop associated to the first row (that is homotopy equivalent to $\gamma_0$) 
and the loop associated to the last row (that is homotopy equivalent to $\gamma_1$). In conclusion, $\gamma_0$ and $\gamma_1$ are homotopy equivalent in $M$.

Since it is easily seen that the correspondance we defined between paths in $M$ and discrete paths in $X$ is surjective and preserves concatenation, it then gives rise to an isomorphism 
between the fundamental groups of $M$ and the discrete fundamental group of $X$, as claimed. 

\end{proof}

Next we give a construction that is similar to suspension of a topological space.
\begin{definition}\label{defin:suspension}
Let $X$ be a metric space, and $r > 0$. Take $Y = X \times \{0, 1, 2, 3 \}$, equipped with the $\ell^1$-metric, and then quotient $Y$ by identifying all points 
whose second coordinate is $0$, and identifying all points whose second coordinate is $3$. Denote this metric space by $X_s$. Let $0$ ($t$, respectively) denote the class 
in $X_s$ corresponding to the points of $Y$ where the second coordinate is $0$ ($3$, respectively).
\end{definition}

\begin{proposition}\label{prop:suspension}
For all $n$, we have $\dhomn{X} \simeq \dhomr{n+1}{X_s}$.
\end{proposition}
\begin{proof}
 Let $A = X_s \setminus \{0 \}$, and $B = X_s \setminus \{t \}$. We claim that $\{A, B \}$ forms a discrete cover. Note that $A \cap B$ is $r$-homotopy equivalent to $X$, and 
$A, B$ are both contractible. The result then follows from the Mayer-Vietoris exact sequence.
\end{proof}

\begin{corollary}\label{cor:groups appearing as homology groups}
 For all abelian group $G$ and for all $\bar n\in\mathbb N$, there is a finite connected simple graph $\Gamma$ such that
$$
\dhom{n}{1}{\Gamma}= \begin{cases} G,\qquad\text{if }n=\bar n  \\ 0,\qquad\text{if }n\leq\bar n \end{cases}
$$
\end{corollary}

\begin{corollary}\label{cor:discrete sphere}
 There is graph $S^n$ with the property
$$
\dhom{k}{1}{S^n}= \begin{cases} \mathbb{Z},\qquad\text{if }k= n  \\ 0,\qquad\text{if }k\neq n \end{cases}
$$
\end{corollary}

\begin{remark}
 Let $X, Y$ be \emph{extended} metric spaces. This means that $d_X: X \times X \mapsto [0, \infty]$ and similarly fro $d_Y$. Then we can define an extended metric space on $X \sqcup Y$:
\[ d_{X \sqcup Y}(a,b) = \left \{ \begin{array}{cc} d_X(a,b) & a, b \in X \\ d_Y(a,b) & a,b \in Y \\ \infty & \end{array} \right. \]

Then clearly $\dhomn{X \sqcup Y} \simeq \dhomn{X} \times \dhomn{Y}$. For any sequence $G_i$ of abelian groups, there is an extended metric space $X$ such that $\dhomn{X} \simeq G_n$. 
The construction uses disjoint union (which is why we must work with extended metric spaces).

\end{remark}

\section{A new coarse homology theory}\label{se:coarse}

As mentioned in the Introduction, the main motivation to work with metric spaces is to construct a new coarse invariant. Since a typical coarse object is a sequence of finite objects, 
the idea is to take the direct limit, as $r$ goes to infinity, of the countably infinite direct product of the discrete homology theory.

To make this idea formal, we first recall some basic facts about direct limits of groups. A direct system of groups is a net $(G_r, \phi_{r,s})_{r,s\in R, r\leq s}$, 
where $R$ is a directed set, $G_r$ are groups, and $\phi_{r,s}:G_r\to G_s$ are group homomorphisms such that $\phi_{s,t}\circ\phi_{r,s}=\phi_{r,t}$, for all $r\leq s\leq t$, and $\phi_{r,r}=Id_{G_r}$, 
for all $r$. In our case, the directed set $R$ is an unbounded interval of positive real numbers $R=[r_0,\infty)$.

\begin{definition}\label{defin:homomorphism of direct systems}
 Let $(G_r,\phi_{r,s})$ and $(H_r,\psi_{r,s})$ be two direct systems of groups. A homomorphism $\alpha:(G_r,\phi_{r,s})\to(H_r,\psi_{r,s})$ of direct systems is a net of group homomorphisms $\alpha_r:G_r\to H_{\rho_r}$ such that
\begin{itemize}
\item the function $r\to\rho_r$ is increasing,
\item for all $r\leq s$, one has
$$
\psi_{\rho_r,\rho_s}\circ\alpha_r=\alpha_s\circ\phi_{r,s}
$$
\end{itemize} 
\end{definition}

Let $(G_r,\phi_{r,s})$ be a direct system of groups. The direct limit $G=\underrightarrow\lim G_r$ is the 
quotient of the disjoint union of all the $G_r$'s modded out by the following equivalence relation: 
$g\in G_r$ is equivalent to $h\in G_s$ if there is $t\geq r,s$ such that $\phi_{r,t}(g)=\phi_{s,t}(h)$.

\begin{remark}\label{rem:directsystemhomomorphism}
{\rm  Let $(G_r,\phi_{r,s})$ and $(H_r,\psi_{r,s})$ be two direct systems of groups and let $G$ and $H$ be respectively their 
direct limits. A homomorphism $\alpha:(G_r,\phi_{r,s})\to(H_r,\psi_{r,s})$ of direct systems induces a canonical group homomorphism 
$\alpha_*:G\to H$ as follows. Given $[g]_{\sim}\in G$, assume that the representative $g$ belongs to $G_r$, and define $\alpha_*([g]_{\sim})=[\alpha_r(g)]_{\sim}$. 
One easily checks that $\alpha_*$ is well defined and it is in fact a group homomorphism from $G$ to $H$.}
\end{remark}

Given two homomorphisms $\alpha:(G_r,\phi_{r,s})\to(H_r,\psi_{r,s})$ and $\beta:(H_r,\psi_{r,s})\to(K_r,\chi_{r,s})$, the composition $\beta\circ\alpha:(G_r,\phi_{r,s})\to(K_r,\chi_{r,s})$ 
is defined
through the net of homomorphisms $(\beta\circ\alpha)_r:=\beta_{\rho_r}\circ\alpha_r$, and 
is a homomorphism of direct systems of groups. 
The following also holds:
\begin{align}\label{eq:composition}
(\beta\circ\alpha)_*=\beta_*\circ\alpha_*.
\end{align}

We need the following Lemma:
\begin{lemma}\label{lem:directlimitisomorphism}
Let $\alpha, \beta:(G_r,\phi_{r,s})\to(H_r,\psi_{r,s})$ be two homomorphisms and suppose there exists $r \in R$, such that for all $s \geq r$, we have 
$\alpha_s = \beta_s$. Then $\alpha_* = \beta_*$.
\end{lemma}
\begin{proof}
 
Let $G, H$ be the direct limits of $(G_r,\phi_{r,s})$ and $(H_r,\psi_{r,s})$ respectively, and let $[g]_\sim \in G$. Without loss of generality, suppose we have a 
representative $g \in G_s$ with $s \geq r$. Then $\alpha_s(g) = \beta_s(g)$, so we have $\alpha_*[g]_\sim = [\alpha_s(g)]_\sim = [\beta_s(g)]_\sim = \beta_*[g]_\sim$.
\end{proof}

A metric space is \emph{proper} if closed balls are compact.

\begin{definition}\label{def:UBP}
 Let $S:(0,\infty)\to(0,\infty)$ be a function. A map $f:X\to Y$ between proper metric spaces is \emph{uniformly S-bornologous} if for all $r>0$ one has
$$
d_X(x_1,x_2)<r\qquad\qquad\text{implies}\qquad\qquad d_Y(f(x_1),f(x_2))<S(r).
$$
A map $f:X\to Y$ is called uniformly bornologous if there is a function $S$ such that $f$ is uniformly $S$-bornologous. 
\end{definition}

We recall that a map $f:X\to Y$ between proper metric space is called \emph{proper} if inverse images of compact sets are compact. Let us introduce some acronyms:
\begin{itemize}
\item $UBP(X,Y)$ is the set of uniformly bornologous and proper maps $f:X\to Y$.
\item $UBP_S(X,Y)$ is the set of uniformly $S$-bornologous and proper maps $f:X\to Y$.
\end{itemize}

\begin{definition}\label{defin:bornotopic}
 Let $f,g\in UBP(X,Y)$. They are \emph{bornotopy equivalent} if there is $R>0$ such that
$$
d_Y(f(x),g(x))<R\qquad\qquad\text{for all }x\in X
$$ 
\end{definition}

\begin{definition}\label{defin:bornotopicequivalence}
 A map $f\in UBP(X,Y)$ is a \emph{bornotopy equivalence} if there is $g\in UBP(Y,X)$ such that $f\circ g$ is bornotopy equivalent to
 Id$_Y$ and $g\circ f$ is bornotopy equivalent to $\id_X$. Two spaces are called bornotopy equivalent if there is a bornotopy equivalence between them. 
\end{definition}

We said at the beginning of this section that coarse geometry is, informally speaking, the study of a metric space from a large scale point of view. 
Formally, coarse geometry is the study of the category whose objects are proper metric spaces and whose morphisms are bornologous maps. Therefore, 
to define a coarse homology theory, we need an homology theory for proper metric spaces that are \emph{coarsely connected} and such that this theory is invariant under bornotopy equivalence.

\begin{definition}\label{defin:coarsely connected}
{\rm A metric space is coarsely connected if it is connected at some scale $r>0$.}
\end{definition}

Throughout the rest of this section, $X$ will be always assumed coarsely connected and the index $r$ will be always assumed to be large enough such that $X$ is connected at scale $r$.

The idea to coarsen the discrete homology theory is very simple: a coarse chain will be a sequence of discrete chains, modded out by bounded sets. Formally, we first need to take the 
direct product of countably many copies of $\dhomn{X}$ and then the direct limit over $r$. So, first set $\wdhomn{X}:=\Pi_{k=1}^{\infty}\dhomn{X}$. For $r\leq s$, denote $\iota_{n,r,s}:\dhomn{X}\to \dhom{n}{s}{X}$ 
the inclusion and let $\tilde\iota_{n,r,s}=(\iota_{n,r,s})_k: \wdhomn{X}\to \wdhom{n}{s}{X}$ be the component-wise homomorphism. Clearly, for every fixed $n$, $(\wdhomn{X},\tilde\iota_{n,r,s})$ is a 
direct system of groups and so we can define

\begin{align}\label{eq:coarse homology}
\dhomi{X}:=\underrightarrow\lim \wdhomn{X}.
\end{align}

\begin{examples}
{\rm 
\begin{enumerate}
\item Every bounded metric space $X$ has trivial $\dhomi{X}$, coherently with the fact that bounded metric spaces are bornotopy equivalent to one point and so they are trivial from a coarse point of view.
\item An example of a metric space $X$ with non trivial $\dhom{1}{\infty}{X}$ is the coarse Hawaiian earring. This space, whose suggestive name has been advised by Nick Wright, is the following coarse analogue of the 
standard Hawaiian earring. Consider the space $X=\bigcup X_n$, where $X_n$ is the circle in the Euclidean plane with center in the point $(0,n^2)$ and passing through the origin of the plane. 
The space $X$, equipped with the subspace-metric inherited by $\mathbb R^2$ is easily seen to have non trivial $\dhom{1}{\infty}{X}$.
\end{enumerate}
}
\end{examples}

Our aim is to prove the following

\begin{theorem}\label{th:invariance}
If $X,Y$ are coarsely connected proper metric spaces which are bornotopy equivalent through the map $f:X\to Y$, then 
$$
\dhomi{X} \simeq \dhomi{Y},\qquad\qquad\text{for all n,}
$$
and the isomorphism is given by a canonical map $f_*$.
\end{theorem}

We start with a simple observation. Let $f$ be a uniformly $S$-bornologous map from $X$ to $Y$. Fix $r$ large enough such that $X$ is connected at 
scale $r$. Since $f$ is uniformly $S$-bornologous, the image under $f$ of any singular $(n,r)$-cube in $X$ is a singular $(n,S(r))$-cube in $Y$. Also, 
$f$ maps degenerate cubes to degenerate cube and therefore $f$ drops to a quotient homomorphism $\overline f_{n,r}:C_{n,r}(X)\to C_{n,S(r)}(Y)$. 
This homomorphism maps $ker(\partial_n)\subseteq C_{n,r}(X)$ to $\ker(\partial_n)\subseteq C_{n,S(r)}(Y)$ and $\text{Im}(\partial_{n+1})\subseteq C_{n,r}(X)$ 
to $\text{Im}(\partial_{n+1})\subseteq C_{n,S(r)}(Y)$. Therefore, for large $r$ we have group homomorphisms $f_{n,r}:\dhomn{X}\to \dhom{n}{S(r)}{Y}$, which induce component-wise 
homorphisms $\tilde{f}_{n,r}:\wdhomn{X}\to\wdhom{n}{S(r)}{Y}$. This net of homomorphisms is clearly an homomorphism of direct systems and therefore 
we can apply Remark \ref{rem:directsystemhomomorphism} and get the following proposition.

\begin{proposition}\label{prop:UBPinduceshomo}
Let $f:X\to Y$ be a UBP map between coarsely connected metric spaces. Then, for all $n$, the natural map $f_{n,*}:\dhomi{X} \to \dhomi{Y}$ is a group homomorphism.
\end{proposition}

Now let $f\in UBP_{S_f}(X,Y)$ and $g\in UBP_{S_g}(X,Y)$ be bornotopy equivalent maps. Let $r>0$ and denote $M_r=\max\{S_f(r),S_g(r)\}$. As before, we have systems of group homomorphisms
$$
f_{n,r}:\dhomn{X} \to \dhom{n}{M_r}{Y}\qquad\qquad\text{and}\qquad\qquad g_{n,r}: \dhomn{X} \to \dhom{n}{M_r}{Y}
$$
induced by $f$ and $g$ (Notice that one of the two may have image inside some $H_{n,s}$ with $s<M_r$. In this case we may pass to $\dhom{n}{M_r}{Y}$ by composing with $\iota_{s,M_r}$). 
Observe that for large enough $r$, the maps $f_{n,r}$ and $g_{n,r}$ are the same, since $f$ and $g$ are bornotopy equivalent, that is their images are point-wise at uniformly bounded distance. 
Thus, by Lemma \ref{lem:directlimitisomorphism}, we have the following:
\begin{proposition}\label{prop:samehomology}
 Let $f, g: X \to Y$ be bornotopy equivalent maps. Then they induce the same map on homology.
\end{proposition}

We can now prove Theorem \ref{th:invariance}.

\begin{proof}[Proof of Theorem \ref{th:invariance}]
Assume $f:X\to Y$ is a bornotopy equivalence and let $g:Y\to X$ its bornologous inverse, that is, $g\circ f\sim\id_X$ and $f\circ g\sim\id_Y$. By Proposition \ref{prop:samehomology} 
this tells us that $f \circ g$ and $\id_Y$ induce the same map on homology, that is, $(f \circ g)_* = \id_{\dhomr{\infty}{X}}$ and $(g \circ f)_* = \id_{\dhomr{\infty}{Y}}$. By Equation \ref{eq:composition}, we then get $f_* \circ g_* = \id_{\dhomr{\infty}{X}}$ and $g_* \circ f_* = \id_{\dhomr{\infty}{Y}}$, that shows that $f_*$ and $g_*$ are both isomorphisms.
\end{proof}

\section{Conclusions and Open Problems}\label{se:conclusion}

In this paper, we discussed a discrete homology theory for metric spaces, related to discrete homotopy theory. We proved some analogues of classical results. 
We also constructed a new coarse homology theory, and showed its relationship to bornotopy equivalence. However, there are many future directions left to go. 
We list a few open questions below.

\begin{enumerate}
\item We proved the Hurewicz theorem in dimension 1. It would interesting to see if the theorem holds also in higher dimension.
\item We proved the \emph{discretization theorem} (Theorem \ref{th:discretization}) in dimension 1 and only for the fundamental group. Also in this case, 
it would be interesting to see if an analogous theorem holds in higher dimension and for both homotopy and homology groups.
\item Let $(X,d)$ be a path-connected metric space. Observe that, consequently, $X$ is connected at scale $r$, for all $r>0$. 
Therefore, we may consider both the standard fundamental group of $X$, $\pi_1(X)$, and the net of discrete fundamental groups $\{A_{1,r}(X)\}_{r>0}$. 
Is there any relation between $\pi_1(X)$ and some sort of limit of the $A_{1,r}(X)$'s, as $r\to0$? Of course, one should first clarify which notion of limit would be suitable. 
Note that, as $r$ decreases, the net $A_{1,r}(X)$ forms an inverse system of groups. So, perhaps, the most natural notion of limit would be that of inverse limit. 
What makes this problem more interesting is the fact that $\pi_1(X)$ is not isomorphic to an inverse limit of $\{A_{1,r}(X)\}_{r > 0}$ in general. Consider $\mathbb R^2\setminus\{(0,0)\}$ equipped with the 
Euclidean metric. Then $A_{1,r}(X)=\{0\}$, for all $r>0$, but $\pi_1(X)=\mathbb Z$. 
\item An interesting problem would be that of finding a coarse analogue of the excision theorem.
\item Many other open problems would come from comparing our coarse homology theory with previously known coarse homology theories or with other known asymptotic concepts, such as asymptotic dimension, asymptotic cones, etc.
\end{enumerate}

\end{document}